\documentclass[11pt]{jgcc}
\usepackage{amssymb}
\usepackage{amsrefs}
\usepackage{mathrsfs}
\usepackage{faktor}
\usepackage{tikz}
\usepackage{thmtools}
\usepackage{thm-restate}
\usetikzlibrary{arrows}
\usetikzlibrary{shapes.geometric}
\usetikzlibrary{decorations.pathreplacing, decorations.markings}

\usepackage{hyperref}
\usepackage{graphicx}

\DeclareMathOperator{\shortmod}{mod}
\DeclareMathOperator{\expsum}{expsum}
\DeclareMathOperator{\Aut}{Aut}

\pdfoutput=1
\usepackage{lastpage}
\jgccdoi{14}{1}{2}{9776}  
\jgccheading{}{\pageref{LastPage}}{}{}%
{July~8,~2022}{Oct.~3,~2022}{}    

\begin{document}

\title{Equations in virtually class 2 nilpotent groups}

\author{Alex Levine}

\address{Department of Mathematics, Alan Turing Building, The University
of Manchester, Manchester M13 9PL, UK}

\email{alex.levine@manchester.ac.uk}



\begin{abstract}
  We give an algorithm that decides whether a single equation in a group that is
  virtually a class 2 nilpotent group with a virtually cyclic commutator
  subgroup, such as the Heisenberg group, admits a solution. This generalises
  the work of Duchin, Liang and Shapiro to finite extensions.
\end{abstract}

\maketitle

\section{Introduction}
  Since the 1960s, many papers have discussed algorithms to decide whether or
  not equations in a variety of different classes of groups admit solutions. An
  \textit{equation} with the variable set \(V\) in a group \(G\) has the form,
  \(w = 1\), for some element \(w \in G \ast F(V)\). A first major positive
  result in this area is due to Makanin, during the 1980s, when in a series of
  papers he proved that it is decidable whether a finite system of equations in
  a free group admits a solution \cites{Makanin_systems, Makanin_semigroups,
  Makanin_eqns_free_group}. Since then, Makanin's work has been extended to show
  the decidability of the satisfiability equations in hyperbolic groups,
  solvable Baumslag-Solitar groups, right-angled Artin groups and more
  \cites{dahmani_guirardel, eqns_RAAGs, diophantine_metabelian_grps}.

  Our primary focus in this paper will be single equations in virtually finitely
  generated class \(2\) nilpotent groups, where if \(\mathcal{P}\) is a property
  of groups, we say a group is \textit{virtually} \(\mathcal{P}\) if it has a
  finite-index subgroup with \(\mathcal{P}\). The \textit{single equation
  problem} in a group \(G\) is the decision question as to whether there is an
  algorithm for \(G\) that takes as input an equation \(w = 1\) in \(G\) and
  outputs whether or not \(w = 1\) admits a solution. Duchin, Liang and Shapiro
  proved that the single equation problem in finitely generated class \(2\)
  nilpotent groups with a virtually cyclic commutator subgroup is decidable
  \cite{duchin_liang_shapiro}. This is in contrast with the fact that the
  satisfiability of systems of equations in free nilpotent groups of class \(2\)
  is undecidable. The assumption that the commutator subgroup is virtually
  cyclic cannot be completely removed; Roman'kov gave an example of a finitely
  generated class \(2\) nilpotent group where it is undecidable whether
  equations of the form \([X_1, \ X_2] = g\), where \(g\) is a constant, admit
  solutions \cite{Romankov_commutators}.

  Our main result is to generalise Duchin, Liang and Shapiro's result to show
  that the single equation problem in virtually a finitely generated class \(2\)
  nilpotent group with a virtually cyclic commutator subgroup is decidable. This
  class includes the Heisenberg group and all higher Heisenberg groups.

  \begin{restatable}{thm}{mainthm}
    \label{virt_class_2_decidable_thm}
     The single equation problem in virtually a group that is class \(2\)
     nilpotent group with a virtually cyclic commutator subgroup is decidable.
  \end{restatable}

  Roman'kov began the study of equations in nilpotent groups, when in 1977 he
  showed that in free nilpotent groups of class at least \(9\) and sufficiently
  large rank, it is undecidable whether finite systems of equations admit a
  solution \cite{romankov_undecidable_first}. Following this, Repin proved that
  there is a finitely presented nilpotent group such that the satisfiability of
  single equations with one variable are undecidable \cite{repin83}, and
  improved this by showing that there are such groups of nilpotency class \(3\)
  \cite{repin}. These results constrast with the fact that the conjugacy
  problem, which is a specific example of a one-variable equation, is decidable
  in all finitely generated nilpotent groups. Repin also showed that the
  satisfiability of single equations with one variable in non-abelian free
  nilpotent groups of class at least \(10^{20}\) is undecidable \cite{repin}.

  In the positive direction, Repin proved that the satisfiability of single
  equations with one variable are decidable in any finitely generated class
  \(2\) nilpotent group \cite{repin}. In addition, Truss showed that the
  satisfiability of single equations in two variables in the free nilpotent
  group of class \(2\) and rank \(2\) (the Heisenberg group) are decidable
  \cite{truss}.


  We prove our main result using a similar method to the method used by Duchin,
  Liang and Shapiro \cite{duchin_liang_shapiro}; by converting an equation in a
  virtually class \(2\) nilpotent group with a virtually cyclic commutator into
  an equivalent system of linear and quadratic equations and congruences in the
  ring of integers. We then show that the system obtained is of the same type as
  that obtained from an equation in a class \(2\) nilpotent group with a
  virtually cyclic commutator subgroup, and is thus decidable using the work of
  Duchin, Liang and Shapiro.

  In Section \ref{prelims_sec}, we define a group equation and solution, and
  give some background on nilpotent groups. In Section
  \ref{transforming_over_Z_sec}, we use the arguments of Duchin, Liang and
  Shapiro \cite{duchin_liang_shapiro} to detail the reduction from a single
  equation in a class \(2\) nilpotent group to a system of equations in the ring
  of integers. We conclude in Section \ref{finite_ext_sec} by using this
  reduction to prove Theorem \ref{virt_class_2_decidable_thm}.

\section{Preliminaries}
  \label{prelims_sec}

  \begin{nota}
    We introduce notation we will frequently use.
    \begin{enumerate}
      \item If \(S\) is a subset of a group, we define \(S^\pm = S \cup
      S^{-1}\). Moreover, if \(a \in S\) and \(w \) is a word over \(S^\pm\)
      we define \(\expsum_a{w}\) to be the number of occurences of \(a\) in
      \(w\) minus the number of occurences of \(a^{-1}\) in \(w\);
      \item For elements \(g\) and \(h\) of a group \(G\), the commutator is
      defined by \([g, \ h] = g^{-1}h^{-1}gh\);
      \item If \(x \in \mathbb{R}\), we will define the floor notation \(\lfloor
      x \rfloor\) in a non-standard way:
      \[
        \lfloor x \rfloor = \left\{
        \begin{array}{cl}
          \max\{y \in \mathbb{Z} \mid y \leq x\} & x \geq 0 \\
          \min\{y \in \mathbb{Z} \mid y \geq x\} & x < 0.
        \end{array}
        \right.
      \]
      That is, we round towards zero.
    \end{enumerate}
  \end{nota}

  \subsection*{Group equations}
  We start with the definition and some examples of equations in groups.

  \begin{defi}
    Let \(G\) be a finitely generated group, \(V\) be a finite set disjoint with
    \(G\), and \(F(V)\) be the free group on \(V\). An \textit{equation} in
    \(G\) is an identity \(w = 1\), where \(w \in G \ast F(V)\). A
    \textit{solution} to \(w = 1\) is a homomorphism \(\phi \colon G \ast F(V)
    \to G\) that fixes elements of \(G\), such that \(\phi(w) = 1\). The
    elements of \(V\) are called the \textit{variables} of the equation. A
    \textit{system of equations} in \(G\) is a finite set of equations in \(G\),
    and a \textit{solution} to a system is a homomorphism that is a solution to
    every equation in the system.

    We say two systems of equations in \(G\) are \textit{equivalent} if they
    have the same set of solutions.

    The \textit{single equation problem} in \(G\) is the decidability question
    as to whether there is an algorithm that accepts as input an equation \(w =
    1\) in \(G\), where the elements of \(G\) within \(w = 1\) are represented
    by words over a finite generating set, and returns \textsc{yes} if \(w = 1\)
    admits a solution and \textsc{no} otherwise.
  \end{defi}

  \begin{rem}
    We will often write a solution to an equation in a group \(G\) as a tuple of
    elements \((g_1, \ \ldots, \ g_n)\), rather than a homomorphism. We can
    recover such a homomorphism \(\phi\) from a tuple by setting \(\phi(g) = g\)
    for each \(g \in G\) and \(\phi(X_i) = g_i\) for each \(X_i\), and the
    action of \(\phi\) on the remaining elements is now determined as it is a
    homomorphism.
  \end{rem}

  \begin{exa}
    Deciding whether an equation in the group \(\mathbb{Z}\) admits a solution
    reduces to solving a linear equation in integers. For example, using the
    free generator \(a\) for \(\mathbb{Z}\),
    \[
      X^2 a^3 Y^2 a^{-3} Y^{-1} a = 1
    \]
    is an equation, which we can rewrite using additive notation as
    \[
      2X + 3 + 2Y - 3 - Y + 1 = 0.
    \]
    We can use the fact that \(\mathbb{Z}\) is abelian to show that this is
    equivalent to \(2X + Y + 1 = 0\), which is just a linear equation in
    integers, and elementary linear algebra can be used to decide if it admits a
    solution (and `construct' the set of solutions). In this case, the equation
    does admit solutions, and the set of solutions is
    \[
      \{(x, \ -2x - 1) \mid x \in \mathbb{Z}\}.
    \]
  \end{exa}

  \subsection*{Nilpotent groups}

  Below we give the definition of a nilpotent group, along with an elementary
  lemma about commutators we will use later on.

  \begin{defi}
    Let \(G\) be a group. Define \(\gamma_i(G)\) for all \(i \in
    \mathbb{Z}_{\geq 0}\) inductively as follows.
    \begin{align*}
      & \gamma_0(G) = G \\
      & \gamma_i(G) = [G, \ \gamma_{i - 1}(G)] \text{ for } i \geq 1.
    \end{align*}
    The subnormal series \((\gamma_i(G))\) is called the \textit{lower central
    series} of \(G\). We call \(G\) \textit{nilpotent} of \textit{class} \(c\)
    if \(\gamma_c(G)\) is trivial.
  \end{defi}

  \begin{lem}[\cite{EDT0L_Heisenberg}, Lemma 2.3]
    \label{commutator_lem}
    Let \(G\) be a class \(2\) nilpotent group, and \(g, \ h \in G\). Then
    \begin{enumerate}
      \item[1.] \([g^{-1}, \ h^{-1}] = [g, \ h]\),
      \item[2.] \([g^{-1}, \ h] = [g, \ h]^{-1}\).
    \end{enumerate}
  \end{lem}



  We now introduce the normal form we will be using for class \(2\) nilpotent
  groups. This is used in \cite{duchin_liang_shapiro}, and we include the proof
  of uniqueness and existence for completeness.

  The following lemma is used to define the Mal'cev generating set and normal
  form.

  \begin{lem}
    \label{malcev_gen_set_lem}
    Let \(G\) be a class 2 nilpotent group with a virtually cyclic
    commutator subgroup. Then \(G\) has a generating set
    \[
      \{a_1, \ \ldots, \ a_n, \ b_1, \ \ldots, \ b_r, \ c, \
      d_1, \ \ldots, \ d_t\},
    \]
    where \(n, \ r, \ t \in \mathbb{Z}_{> 0}\), such that the \(d_i\)s have
    finite order, \(c\) and the \(d_i\)s are central, for each \(b_i\), there
    exists \(l_i \in \mathbb{Z}_{> 0}\), such that \(b_i^{l_i} \in [G, \ G]\),
    and \([G, \ G] = \langle c, \ d_1, \ \ldots, \ d_t \rangle\).

    Moreover, every element of \(G\) can be expressed uniquely as an
    element of the set
    \begin{align}
      \{ & a_1^{i_1} \cdots a_n^{i_n} b_1^{j_1} \cdots b_r^{j_r}
      c^p d_1^{q_1} \cdots d_t^{q_t} \mid
      i_1, \ \ldots, \ i_n, \ p \in \mathbb{Z}, \\
      \nonumber
      & j_x \in \{0, \
      \ldots, l_x - 1\}, \ q_x \in \{0, \ \ldots, \ k_x - 1\} \text{ for each }
      x\}.
    \end{align}
  \end{lem}

  \begin{proof}
    Using the fundamental theorem for finitely generated abelian groups and the
    fact that \([G, \ G]\) is virtually cyclic, the short exact sequence
    \(\{1\} \to [G, \ G] \to G \to \faktor{G}{[G, \ G]} \to \{1\}\) becomes
    \[
    \{1\} \longrightarrow \mathbb{Z} \oplus (\mathbb{Z}_{k_1} \oplus \cdots
    \oplus \mathbb{Z}_{k_t}) \longrightarrow G \longrightarrow \mathbb{Z}^n
    \oplus (\mathbb{Z}_{l_1} \oplus \cdots \oplus \mathbb{Z}_{l_r})
    \longrightarrow \{1\},
    \]
    where \(n \in \mathbb{Z}_{>0}\)  and \(r, \ t \in \mathbb{Z}_{\geq 0}\). Let
    \(a_1, \ \ldots, \ a_n\) be lifts in \(G\) of standard generators for
    \(\mathbb{Z}^n\), \(b_1, \ \ldots, \ b_r\) be lifts of generators of
    \(\mathbb{Z}_{l_1}, \ \ldots, \ \mathbb{Z}_{l_r}\), respectively. Let \(c\)
    be a generator for \(\mathbb{Z}\), and \(d_1, \ \ldots, \ d_t\) be
    generators of \(\mathbb{Z}_{k_1}, \ \ldots, \ \mathbb{Z}_{k_t}\),
    respectively. Then using our short exact sequence, it follows that \(\{a_1,
    \ \ldots, \ a_n, \ b_1, \ \ldots, \ b_r, \ c, \ d_1, \ \ldots, \ d_t\}\)
    generates \(G\). We have that \(d_i^{k_i} = 1\), for all \(i\). As \(\{c, \
    d_1, \ \ldots, \ d_t\}\) generates \([G, \ G]\), the and we have shown that
    the generating set exists.

    We now turn our attention to the normal form, showing existence and
    uniqueness.

    Existence: Let \(g \in G\), and \(w\) be a word over our generating set that
    represents \(g\). As \(c\) and all \(d_i\)s are central, we can push them to
    the back of \(w\), and into the desired order. As \([a_i, \ a_j], \ [b_i, \
    b_j]\), and \([a_i, \ b_j]\) can be written as expressions using \(c\) the
    and \(d_i\)s, we have that reordering the \(a_i\)s and \(b_i\)s to the
    desired form simply creates expressions using \(c\)s and the \(d_i\)s, which
    can then be pushed to the back of \(w\), and into the stated order. Let \(i
    \in \{1, \ \ldots, \ r\}\). By definition, \([G, \ G] b_i^{l_i} = [G, \
    G]\), so we can reduce \(b_i\) modulo \(l_i\) by creating an expression over
    \(c\) and the \(d_i\)s, which, again, can be pushed to the back and into the
    desired form. Since the \(d_i\)s have finite order, we can reduce their
    exponents modulo these orders.

    Uniqueness: Let \(i_1, \ \ldots, \ i_{n + r + 1 + t} \in \mathbb{Z}\) and
    \(j_1, \ \ldots, \ j_{n + r + 1 + t} \in \mathbb{Z}\) be such that
    \[
      u = a_1^{i_1} \cdots a_n^{i_n} b_1^{i_{n + 1}} \cdots b_r^{i_{n + r}} c^{i_{
      n + r + 1}} d_1^{i_{n + r + 1 + 1}} \cdots
      d_t^{i_{n + r + 1 + t}}
    \]
    and
    \[
      v = a_1^{j_1} \cdots a_n^{j_n} b_1^{j_{n + 1}} \cdots b_r^{j_{n + r}} c^{j_{
      n + r + 1}} d_1^{j_{n + r + 1 + 1}} \cdots
      d_t^{j_{n + r + 1 + t}}
    \]
    and expressions in the normal form stated in the lemma. Suppose \(u =_G v\).
    Then \(u\) and \(v\) have the same image in the quotient of \(G\) by \([G, \
    G]\), and so \(a_1^{i_1} \cdots a_n^{i_n} b_1^{i_{n + 1}} \cdots b_r^{i_{n +
    r}} =_{G / [G, \ G]} a_1^{j_1} \cdots a_n^{j_n} b_1^{j_{n + 1}} \cdots
    b_r^{j_{n + r}}\). As these words are in the standard normal form for these
    finitely generated abelian groups, it follows that \(i_x = j_x\) for all \(x
    \in \{1, \ \ldots, \ n + r\}\). Thus
    \[
      c^{i_{ n + r + 1}} d_1^{i_{n + r + 1 + 1}} \cdots d_t^{i_{n + r + 1 + t}}
      =_G c^{j_{ n + r + 1}} d_1^{j_{n + r + 1 + 1}} \cdots d_t^{j_{n + r + 1 +
      t}}.
    \]
    As \([G, \ G]\) is abelian and these words are in the standard normal form
    for \([G, \ G]\), it follows that all of the exponents in these expressions
    must be equal; that is \(i_x = j_x\) for all \(x \in \{n + r + 1, \ \ldots,
    \ n + r + 1 + t\}\). We have now shown that \(u\) and \(v\) are equal as
    words, as required.

  \end{proof}

  \begin{defi}
    A generating set defined as in Lemma \ref{malcev_gen_set_lem} is called a
    \textit{Mal'cev generating set}, and the normal form defined in Lemma
    \ref{malcev_gen_set_lem} is called the \textit{Mal'cev normal form}.
  \end{defi}

  As we have seen in the proofs of the previous lemma, one can manipulate
  words in class \(2\) nilpotent groups with a virtually cyclic commutator
  subgroup by pushing \(a_i\)s past \(b_i\)s and paying a `cost' in \(c\)
  and the \(d_i\)s. Quantifying the `cost' for each such move will be
  necessary to convert a given equation in a class \(2\) nilpotent group into
  a system of equations in the ring of integers, with the `cost' appearing as
  constants in this system.

  \begin{nota}
    \label{commutator_exponent_ntn}
    We define a number of values for a group \(G\) with the Mal'cev
    generating set
    \[
      \{a_1, \ \ldots, \ a_n, \ b_1, \ \ldots, \ b_r, \ c, \
      d_1, \ \ldots, \ d_t\},
    \]
    where again, \(l_i\) is minimal (and exists) such that \(b_i^{l_i} \in
    [G, \ G]\) and the order of \(d_i\) is \(k_i\).
    \begin{enumerate}
      \item From Lemma \ref{malcev_gen_set_lem}, we have that \([a_i, \ a_j] , \
      [b_i, \ b_j], \ [a_i, \ b_j] \in \{c^p d_1^{q_1} \cdots d_t^{q_t} \mid p,
      \ q_1, \ \ldots, \ q_t \in \mathbb{Z}\}\), for all \(i, \ j\), with \(i <
      j\) in the first two expressions. For all such \(i\) and \(j\), \(k \in
      \{1, \ \ldots, \ s\}\), and \(l \in \{1, \ \ldots, \ t\}\), we can
      therefore define \(\alpha_{ijl}\), \(\beta_{ijl}\) and \(\gamma_{ijl}\) to
      be the unique integers satisfying the following normal form expressions in
      \([G, \ G]\):
      \begin{align*}
        [a_j, \ a_i] & = c^{\alpha_{ij0}}
        d_1^{\alpha_{ij1}} \cdots d_t^{\alpha_{ijt}}, \quad (i < j)
        \\
        [b_j, \ b_i] & = c^{\beta_{ij0}}
        d_1^{\beta_{ij1}} \cdots d_t^{\beta_{ijt}}, \quad (i < j) \\
        [a_i, \ b_j] & = c^{\gamma_{ij0}}
        d_1^{\gamma_{ij1}} \cdots d_t^{\gamma_{ijt}}.
      \end{align*}
      \item Since \(b_i^{l_i} \in [G, \ G]\), we can define \(\eta_{il}\) for
      all \(i \in \{1, \ \ldots, \ r\}\), and \(l \in \{0, \ \ldots, \ t\}\), to
      be the unique integers such that
      \[
        b_i^{l_i} = c^{\eta_{i0}} d_1^{\eta_{i1}}
        \cdots d_t^{\eta_{it}}.
      \]
    \end{enumerate}
  \end{nota}

\section{Transforming equations in nilpotent groups into equations in integers}
  \label{transforming_over_Z_sec}

    This section aims to prove Lemma \ref{Z_system_lem}; that is, a single
    equation \(\mathcal{E}\) in a class 2 nilpotent group with a virtually
    cyclic commutator subgroup is equivalent to a system \(S_\mathcal{E}\) over
    \(\mathbb{Z}\) of (1) linear equations and congruences, (2) a single
    quadratic equation and (3) quadratic congruences, where the quadratic
    equations and congruences can also contain `floor' terms.

    The idea of the proof is to replace each variable in \(\mathcal{E}\) with a
    word representing a potential solution, and then convert this new word into
    Mal'cev normal form. The linear equations in \(S_\mathcal{E}\) occur as the
    solution to the exponent of each generator \(a_i\) being set to \(0\), and
    the linear congruences, quadratic equation and quadratic congruences occur
    when the same is done for the \(b_i\)s, \(c\) and the \(d_i\)s,
    respectively.
    %

    We begin with an example of this process.

    \begin{exa}
      \label{torsion_and_csts_ex}
      Let \(G\) be the class \(2\) nilpotent group with the presentation
      \begin{align*}
        \langle & a_1, \ a_2, \ b, \ c, \ d \mid c = [a_1, \ a_2], \
        d = [a_1, \ b] = [a_2, \ b], \ b^2 = c, \ d^2 = 1, \\
        &  [a_1, \ c] = [a_1, \ d] =
        [a_2, \ c] = [a_2, \ d] = [b, \ c] = [b, \ d] = 1 \rangle.
      \end{align*}
      Consider the equation
      \begin{align}
        \label{torsion_and_csts_ex_eqn}
        X b a_1 c X a_2 c^{-3} a_1 X  = 1
      \end{align}
      We first transform the constants in this equation into Mal'cev normal
      form, push all the commutators to the right, and then use the relation
      \(d^2 = 1\) to obtain
      \begin{align}
        \label{comms_to_back_eqn}
        X a_1 b X a_1 a_2 X c^{-3} d = 1.
      \end{align}
      We set \(X = a_1^{X_1}
      a_2^{X_2} b^{X_3} c^{X_4} d^{X_5}\) using our Mal'cev normal form, for
      new variables \(X_1, \ldots, \ X_5\) over \(\mathbb{Z}\). Plugging this
      into (\ref{comms_to_back_eqn}) gives
      \begin{align}
        \label{plug_in_torsison_case_eqn}
        a_1^{X_1} a_2^{X_2} b^{X_3} c^{X_4} d^{X_5} a_1 b
        a_1^{X_1} a_2^{X_2} b^{X_3} c^{X_4} d^{X_5} a_1 a_2
        a_1^{X_1} a_2^{X_2} b^{X_3} c^{X_4} d^{X_5} c^{-3} d = 1.
      \end{align}
      We can the transform this into Mal'cev normal form, to (first) obtain
      \begin{align}
        \label{first_malcev_torsion_case_eqn}
          & a_1^{3X_1 + 2} a_2^{3 X_2 + 1} b^{3 X_3 + 1} c^{3 X_4 + X_1(1 +
          X_2 + X_2) + X_1X_2 - 3} \\
          \nonumber
          & d^{3 X_5 + X_2(X_3 + 1 + X_3) + X_1(X_3 + 1
          + X_3) + (X_3 + 1 + X_3) + (X_3 + 1 + X_3) + X_2(1 + X_3) + X_1(1 +
          X_3) + X_3 + 2 + 1} = 1.
      \end{align}
      Simplifying this gives
      \begin{align}
        \label{second_malcev_torsion_case_eqn}
        a_1^{3X_1 + 2} a_2^{3 X_2 + 1} b^{3 X_3 + 1} c^{3 X_1 X_2 + X_1 + 3X_4
        - 3} d^{3 X_1 X_3 + 3 X_2 X_3 + 4X_1 + 4X_2 + 5X_3 + 3X_5 + 5} = 1.
      \end{align}
      Using the relations \(b^2 = c\) and \(d^2 = 1\), we can conclude
      \begin{align}
        \label{third_malcev_torsion_case_eqn}
        a_1^{3X_1 + 2} a_2^{3 X_2 + 1} b^{(X_3 + 1)\shortmod 2} c^{3 X_1 X_2 +
        X_1 + 3X_4 - 3 + \left\lfloor \frac{X_3 + 1}{2} \right\rfloor + X_3}
        d^{(X_1 X_3 + X_2 X_3 + X_3 + X_5 + 1) \shortmod 2} = 1.
      \end{align}
      This results in the following system of equations over (the ring)
      \(\mathbb{Z}\)
      \begin{align}
        \label{torsion_case_Z_system}
        & 3 X_1 + 2 = 0 \\
        \nonumber
        & 3 X_2 + 1 = 0 \\
        \nonumber
        & X_3 + 1 \equiv 0 \mod 2 \\
        \nonumber
        & 3 X_1 X_2 + X_1 + X_3 + \left\lfloor \frac{X_3 + 1}{2} \right\rfloor +
        3 X_4 - 3 = 0 \\
        \nonumber
        & X_1 X_3 + X_2 X_3 + X_3 + X_5 + 1 \equiv 0 \mod 2.
      \end{align}
      As \(3X_1 + 2 = 0\) admits no integer solutions, we can conclude that our
      equation (\ref{torsion_and_csts_ex_eqn}) does not admit a solution.
    \end{exa}

    \begin{nota}
      \label{eqn_ntn}
      Let \(G\) be a class 2 nilpotent group, \(X_1, \ \ldots, \ X_M\) be
      variables, where \(M \in \mathbb{Z}_{> 0}\). Let \(N \in \mathbb{Z}_{>
      0}\), \(\epsilon_1, \ \ldots \epsilon_N \in \{-1, \ 1\}\), and
      \begin{equation}
        \label{class_2_general_eqn}
        \omega_1 X_{p_1}^{\epsilon_1} \cdots \omega_N X_{p_N}^{\epsilon_N} = 1
      \end{equation}
      be an equation over \(G\), where \(\omega_1, \ \ldots, \ \omega_N\) are
      words in Mal'cev normal form, over a Mal'cev generating set for \(G\), as
      constructed in Lemma \ref{malcev_gen_set_lem} and \(p_1, \ \ldots, \
      p_M \in \{1, \ \ldots, \ M\}\). We will also use the
      notation introduced in Lemma \ref{malcev_gen_set_lem} for the generators.
      We will use \(\boldsymbol{\nu}_1, \ \ldots, \ \boldsymbol{\nu}_N\) to be a
      potential solution. We define a number of values based on the
      \(\omega_z\)s and \(\boldsymbol{\nu}_z\)s. To make it clearer, the
      potential solution is shown in bold.

      For each Mal'cev generator \(a\), we will define \(\boldsymbol{\nu}_{z, a}\) and
      \(\omega_{z, a}\) by:
      \begin{align*}
        \boldsymbol{\nu}_{z, a} = \expsum_a(\boldsymbol{\nu}_z), \qquad
        \omega_{z, a} = \expsum_a(\omega_z).
      \end{align*}

      By convention we will often use \(d_0 = c\) and take \(\equiv_{k_0}\) to
      be equality (since \(c\) is infinite order and \(k_0\) is being used to
      represent the order of \(d_0 = c\), this equality modulo \(k_0\)
      is true equality).
    \end{nota}

    This lemma transforms the equation that we obtained from
    (\ref{class_2_general_eqn}) to obtain a system of linear and quadratic
    equations and congruences over the integers. We do this by first
    transforming the equation with the potential solution subbed in into Mal'cev
    normal form. This corresponds to moving from
    (\ref{first_malcev_torsion_case_eqn}) to
    (\ref{third_malcev_torsion_case_eqn}) in Example \ref{torsion_and_csts_ex}.
    Following that, we equate all of the exponents in this word to zero, given
    our system, which is done to obtain (\ref{torsion_case_Z_system}) in Example
    \ref{torsion_and_csts_ex}, respectively. The capital Latin alphabet
    characters are constants derived from the constants of the equation, and the
    group's structure. Recall that the \(\boldsymbol{\nu}_{z, a}\)s represent
    variables over \(\mathbb{Z}\) (see Notation \ref{eqn_ntn}). For \(i \in \{1,
    \ -1\}\) will use \(\delta_i = 1\) if \(i = -1\) and \(\delta_i = 0\)
    otherwise.

    \begin{lem}
      \label{system_over_Z_technical_lem}
      The words \(\boldsymbol{\nu}_1, \ \ldots, \ \boldsymbol{\nu}_M\) form a
      solution to (\ref{class_2_general_eqn}) in a class \(2\) nilpotent group
      with a  virtually cyclic commutator subgroup, if and only if the following
      equations and congruences hold:
      \begin{align}
        \label{a_m_eqn}
        & A_m + \sum_{z = 1}^N \epsilon_z \boldsymbol{\nu}_{p_z, a_m} = 0, \\
        \label{b_m_eqn}
        & B_m + \sum_{z = 1}^N \epsilon_z \boldsymbol{\nu}_{p_z, b_m} \equiv_{l_m} 0,\\
        \label{d_m_eqn}
        & D_m + \sum_{z = 1}^N \epsilon_z \boldsymbol{\nu}_{p_z, d_m}
        - \mathop{\sum_{z, u = 1}^N}_{u < z} \sum_{j = 1}^n \epsilon_u \boldsymbol{\nu}_{p_u, a_j} K_{mzj}
        + \mathop{\sum_{z, u = 1}^N}_{u < z} \sum_{j = 1}^r \epsilon_u \boldsymbol{\nu}_{p_u, b_j} L_{mzj}
        - \mathop{\sum_{z, u = 1}^N}_{u \leq z} \sum_{i = 1}^n \epsilon_t \boldsymbol{\nu}_{p_t, a_i}
         \\
        & - J_{mui} - \mathop{\sum_{z, u = 1}^N}_{u < z}
        \mathop{\sum_{i, j = 1}^n}_{i < j} \epsilon_z \epsilon_u \boldsymbol{\nu}_{p_z, a_i} \boldsymbol{\nu}_{p_u, a_j}
        \alpha_{ijm}
        - \mathop{\sum_{z, u = 1}^N}_{u < z} \sum_{i = 1}^n \sum_{j = 1}^r
        \epsilon_z \epsilon_u \boldsymbol{\nu}_{p_z, a_i} \boldsymbol{\nu}_{p_u, b_j} \gamma_{ijm}
        \nonumber \\
        & - \mathop{\sum_{z, u = 1}^N}_{u < z}
        \sum_{j = 1}^r \boldsymbol{\nu}_{p_u, b_j} M_{mzj}
        - \mathop{\sum_{z, u = 1}^N}_{u \leq z} \sum_{i = 1}^r \epsilon_z \boldsymbol{\nu}_{p_z, b_i}
        - O_{mui}
        - \mathop{\sum_{z, u = 1}^N}_{u < z}
        \mathop{\sum_{i, j = 1}^r}_{i < j} \epsilon_z \epsilon_u \boldsymbol{\nu}_{p_z, b_i} \boldsymbol{\nu}_{p_u, b_j} \beta_{ijm}
        \nonumber
        \\
        & - \sum_{z = 1}^N \sum_{i = 1}^r \eta_{im} \left \lfloor
        \frac{\omega_{z, b_i} + \epsilon_z \boldsymbol{\nu}_{p_z, b_i}}{l_i} \right \rfloor
        \nonumber \\
        \nonumber
        & - \sum_{z = 1}^N  \delta_{\epsilon_z} \left(
        \mathop{\sum_{i, j = 1}^n}_{i < j} \alpha_{ijm} \boldsymbol{\nu}_{p_z, a_i} \boldsymbol{\nu}_{p_z, a_j}
        + \sum_{i = 1}^n \sum_{j = 1}^r \gamma_{ijm}
        \boldsymbol{\nu}_{p_z, a_i} \boldsymbol{\nu}_{p_z, b_j}
        + \mathop{\sum_{i, j = 1}^r}_{i < j} \beta_{ijm} \boldsymbol{\nu}_{p_z, b_i}
        \boldsymbol{\nu}_{p_z, b_j} \right)
        \equiv_{k_m} 0
      \end{align}
      where for all \(m\)
      \begin{align*}
        A_m & = \sum_{z = 1}^N \omega_{z, a_m},
        & B_m & = \sum_{z = 1}^N \omega_{z, b_m}, \\
        J_{mui} & = \mathop{\sum_{j = 1}^n}_{i < j} \omega_{u, a_j}
        \alpha_{ijm} + \sum_{j = 1}^r \omega_{u, b_j} \gamma_{ijm},
        & K_{mzj} & = \mathop{\sum_{i = 1}^n}_{i < j}
        \omega_{z, a_i} \alpha_{ijm}, \\
        L_{mzj} & = \sum_{i = 1}^n \omega_{z, a_i} \gamma_{ijm},
        & M_{mzj} & = \mathop{\sum_{i = 1}^r}_{i < j}
        \omega_{z, b_i} \beta_{ijm}, \\
        O_{mui} & = \mathop{\sum_{j = 1}^r}_{i < j} \omega_{u, b_j}
        \beta_{ijm},
      \end{align*}

      \begin{align*}
        D_m & = \sum_{z = 1}^N \omega_{z, d_m}
        \\
        & \quad - \mathop{\sum_{z, u = 1}^N}_{u < z}
        \mathop{\sum_{i, j = 1}^n}_{i < j} \omega_{z, a_i} \alpha_{ijm}
        \omega_{u, a_j}
        - \mathop{\sum_{z, u = 1}^N}_{u < z} \sum_{i = 1}^n \sum_{j = 1}^r
        \omega_{z, a_i} \gamma_{ijm} \omega_{u, b_j}
        - \mathop{\sum_{z, u = 1}^N}_{u < z} \mathop{\sum_{i, j = 1}^r}_{i < j}
         \omega_{z, b_i} \beta_{ijm} \omega_{u, b_j}.
      \end{align*}
    \end{lem}

    \begin{proof}
      Let \(w = \omega_1 \boldsymbol{\nu}_{p_1}^{\epsilon_1} \cdots \omega_M
      \boldsymbol{\nu}_{p_M}^{\epsilon_M}\) by the left-hand side of the equation
      \eqref{class_2_general_eqn}, with the potential solution
      \((\boldsymbol{\nu}_1, \ \ldots, \ \boldsymbol{\nu}_N)\) plugged in. By
      pushing the \(d_m\)s (recall that \(c = d_0\)) to the end of \(w\), we
      have that \(w\) now comprises \(2N\) words over \(\{a_1, \ \ldots, \
      b_r\}^\pm\) in Mal'cev normal form followed by an expression of \(d_m\)s.
      We will now convert \(w\) into Mal'cev normal form, in order to compare
      the exponents of the generators of this normal form version for \(w\) to
      \(0\). From now on, whenever we modify \(w\), we will continue to use the
      fact that the \(d_m\)s are central to push them to the right.

      Using Notation \ref{commutator_exponent_ntn}, if \(i < j\), then \(a_j a_i
      = a_i a_j [a_i, \ a_j]^{-1} = a_i a_j d_0^{-\alpha_{ij0}} d_1^{-\alpha_{ij1}}
      \cdots d_t^{-\alpha_{ijt}}\) and \(b_j b_i = b_i b_j d_0^{-\beta_{ij0}}
      d_1^{-\beta_{ij1}} \cdots d_t^{-\beta_{ijt}}\). Similarly, for any \(i\) and
      \(j\), \(b_j a_i = a_i b_j d_0^{-\gamma_{ij0}} d_1^{-\gamma_{ij1}} \cdots
      d_t^{-\gamma_{ijt}}\). We will use this to reorder all of the subwords
      \((a_1^{\boldsymbol{\nu}_{p_z, a_1}} \cdots a_n^{\boldsymbol{\nu}_{p_z,
      a_n}} b_1^{\boldsymbol{\nu}_{p_z, b_1}} \cdots b_r^{\boldsymbol{\nu}_{p_z,
      b_r}})^{\epsilon_z}\) into a word within \((a_1^\pm)^\ast \cdots
      (a_n^\pm)^\ast (b_1^\pm)^\ast \cdots (b_r^\pm)^\ast (d_0^\pm)^\ast
      (d_1^\pm)^\ast \cdots (d_t^\pm)^\ast\), subject to `creating' some
      additional commutators, which are then pushed to the right. Note that if
      \(\epsilon_z = 1\), then the word is already in the desired form, so
      consider when \(\epsilon_z = -1\). Let \(u =
      \boldsymbol{\nu}_z^{\epsilon_z}\) be such a subword (that is, \(\epsilon_z
      = -1\)). Then
      \[
        u = b_r^{-\boldsymbol{\nu}_{p_z, b_r}} \cdots b_1^{-\boldsymbol{\nu}_{p_z, b_1}}
        a_n^{-\boldsymbol{\nu}_{p_z, a_n}} \cdots a_1^{-\boldsymbol{\nu}_{p_z, a_1}}.
      \]
      We will start at the right, and push terms to the left. We have that
      the \(a_1\)s will have to be pushed past everything (except each other),
      the \(a_2\)s will need to be pushed past everything except the \(a_1\)s,
      and so on up to the \(b_{r - 1}\)s, which will only need to be pushed past
      the \(b_r\)s, and the \(b_r\)s which will not need to be pushed past
      anything, as they will now be in the correct place. Thus
      \begin{align*}
        u = & a_1^{-\boldsymbol{\nu}_{p_z, a_1}} \cdots a_n^{-\boldsymbol{\nu}_{p_z, a_n}}
        b_1^{-\boldsymbol{\nu}_{p_z, b_1}} \cdots b_r^{-\boldsymbol{\nu}_{p_z, b_r}} \\
        & \prod_{m = 0}^t d_m^{\displaystyle - \mathop{\sum_{i, j = 1}^n}_{i < j} \alpha_{ijm}
        \boldsymbol{\nu}_{p_z, a_i} \boldsymbol{\nu}_{p_z, a_j}
        - \sum_{i = 1}^n \sum_{j = 1}^r \gamma_{ijm} \boldsymbol{\nu}_{p_z, a_i}
        \boldsymbol{\nu}_{p_z, b_j}
        - \mathop{\sum_{i, j = 1}^r}_{i < j} \beta_{ijm} \boldsymbol{\nu}_{p_z, b_i} \boldsymbol{\nu}_{p_z, b_j}}.
      \end{align*}
      Now consider the general case for \(u = \boldsymbol{\nu}_z^{\epsilon_z}\),
      with \(\epsilon_z \in \{-1, \ 1\}\). We have
      \begin{align*}
        u = & a_1^{\epsilon_z \boldsymbol{\nu}_{p_z, a_1}} \cdots a_n^{\epsilon_z \boldsymbol{\nu}_{p_z,
        a_n}} b_1^{\epsilon_z \boldsymbol{\nu}_{p_z, b_1}} \cdots b_r^{\epsilon_z \boldsymbol{\nu}_{p_z,
        b_r}} \\
        & \prod_{m = 0}^t d_m^{\displaystyle - \delta_{\epsilon_z} \left(
        \mathop{\sum_{i, j = 1}^n}_{i < j} \alpha_{ijm} \boldsymbol{\nu}_{p_z, a_i} \boldsymbol{\nu}_{p_z, a_j}
        + \sum_{i = 1}^n \sum_{j = 1}^r \gamma_{ijm}
        \boldsymbol{\nu}_{p_z, a_i} \boldsymbol{\nu}_{p_z, b_j}
        + \mathop{\sum_{i, j = 1}^r}_{i < j} \beta_{ijm} \boldsymbol{\nu}_{p_z, b_i}
        \boldsymbol{\nu}_{p_z, b_j} \right)}.
      \end{align*}

      We now push all \(a_i\)s to the left, whilst calculating the
      cost in \(d_m\)s. For \(\omega_1\) there is nothing to do. For
      \(\boldsymbol{\nu}_1\), we have \(\boldsymbol{\nu}_{1, a_i}\) \(a_i\)s,
      and we must move each of these past \(\omega_{1, a_j}\) \(a_j\)s (where
      \(j > i\)), and past \(\omega_{1, b_j}\) \(b_j\)s, (where \(j\) is
      arbitrary). So moving the \(a_i\)s to the left (provided all lower indexed
      \(a_i\)s have already been moved) will increase the number of
      \(d_m\)s by
      \[
        - \mathop{\sum_{i, j = 1}^n}_{i < j} \epsilon_1 \boldsymbol{\nu}_{p_1, a_i} \omega_{1, a_j}
        \alpha_{ijm}
        - \sum_{i = 1}^n \sum_{j = 1}^r \epsilon_1 \boldsymbol{\nu}_{p_1, a_i} \omega_{1, b_j} \gamma_{ijm},
      \]
      Doing the same for the \(a_i\)s in \(\omega_2\), we will now
      have to push them past the \(a_j\)s and \(b_j\)s in \(\omega_1\) and
      \(\boldsymbol{\nu}_1\), so this will increase the number of \(d_m\)s by
      \begin{align*}
        & - \mathop{\sum_{i, j = 1}^n}_{i < j} \omega_{2, a_i} \alpha_{ijm}
        (\omega_{1, a_j} + \epsilon_1 \boldsymbol{\nu}_{p_1, a_j} )
        - \sum_{i = 1}^n \sum_{j = 1}^r \omega_{2, a_i} \gamma_{ijm}( \omega_{1,
        b_j} + \epsilon_1 \boldsymbol{\nu}_{p_1, b_j}),
      \end{align*}
      respectively. Proceeding in this manner for the remaining \(\omega_t\)s
      and \(\boldsymbol{\nu}_t\)s gives the total increase of the \(d_m\)s as
      \begin{align*}
        & - \mathop{\sum_{t, u = 1}^N}_{u < t} \left(
        \mathop{\sum_{i, j = 1}^n}_{i < j}
        \omega_{t, a_i} \alpha_{ijm} (\omega_{u, a_j} + \epsilon_u \boldsymbol{\nu}_{p_u, a_j} )
        + \sum_{i = 1}^n \sum_{j = 1}^r \omega_{t, a_i} \gamma_{ijm}
        (\omega_{u, b_j} + \epsilon_u \boldsymbol{\nu}_{p_u, b_j}) \right) \\
        & - \mathop{\sum_{t, u = 1}^N}_{u \leq t} \left(\mathop{\sum_{i, j =
        1}^n}_{i < j} \epsilon_t \boldsymbol{\nu}_{p_t, a_i} \omega_{u, a_j} \alpha_{ijm}
        + \sum_{i = 1}^n \sum_{j = 1}^r \epsilon_t \boldsymbol{\nu}_{p_t, a_i} \omega_{u, b_j}
        \gamma_{ijm} \right) \\
        & - \mathop{\sum_{t, u = 1}^N}_{u < t}
        \left( \mathop{\sum_{i, j = 1}^n}_{i < j} \epsilon_t \epsilon_u \boldsymbol{\nu}_{p_t, a_i} \boldsymbol{\nu}_{p_u, a_j}
        \alpha_{ijm} + \sum_{i = 1}^n \sum_{j = 1}^r \epsilon_t \epsilon_u \boldsymbol{\nu}_{p_t, a_i} \boldsymbol{\nu}_{p_u, b_j}
        \gamma_{ijm} \right).
      \end{align*}
      This occurs in (\ref{first_malcev_torsion_case_eqn}) and
      (\ref{second_malcev_torsion_case_eqn}) in Example
      \ref{torsion_and_csts_ex}. We will now reorder the \(b_i\)s, which occurs
      in (\ref{first_malcev_torsion_case_eqn}) and
      (\ref{second_malcev_torsion_case_eqn}) in Example
      \ref{torsion_and_csts_ex}.  Again, those in \(\omega_1\) are already in
      position, and pushing those in \(\boldsymbol{\nu}_1\) into place increases
      the number of \(d_m\)s by
      \[
        - \mathop{\sum_{i, j = 1}^r}_{i < j} \epsilon_1 \boldsymbol{\nu}_{p_1, b_i}
        \omega_{1, b_j} \beta_{ijm}.
      \]
      Doing the same for \(\omega_2\) increases the number by
      \begin{align*}
        - \mathop{\sum_{i, j = 1}^r}_{i < j} \omega_{2, b_i} \beta_{ijm}
        (\omega_{1, b_j} + \epsilon_1 \boldsymbol{\nu}_{p_1, b_j}).
      \end{align*}
      Doing the same for all \(\omega_z\)s and \(\boldsymbol{\nu}_{p_z}\)s increases
      the exponent sum of the \(d_m\)s by
      \begin{align*}
        - \mathop{\sum_{t, u = 1}^N}_{u < t} \left(
        \mathop{\sum_{i, j = 1}^r}_{i < j}
        \omega_{t, b_i} \beta_{ijm} (\omega_{u, b_j} + \epsilon_u \boldsymbol{\nu}_{p_u, b_j} )\right)
        - \mathop{\sum_{t, u = 1}^N}_{u \leq t}
        \mathop{\sum_{i, j = 1}^r}_{i < j} \epsilon_t \boldsymbol{\nu}_{p_t, b_i} \omega_{u, b_j}
        \beta_{ijm}
        - \mathop{\sum_{t, u = 1}^N}_{u < t} \mathop{\sum_{i, j = 1}^r}_{i < j}
        \epsilon_t \epsilon_u \boldsymbol{\nu}_{p_t, b_i} \boldsymbol{\nu}_{p_u, b_j} \beta_{ijm}.
      \end{align*}
      It remains to reduce the \(b_i\)s with respect to their modularities, as
      is done in (\ref{third_malcev_torsion_case_eqn}) in Example
      \ref{torsion_and_csts_ex}. We have that doing so increases the number of
      \(d_m\)s by
      \[
        - \sum_{t = 1}^N \sum_{i = 1}^r \eta_{im} \left \lfloor
        \frac{\omega_{t, b_i} + \epsilon_t \boldsymbol{\nu}_{p_t, b_i}}{l_i} \right \rfloor,
      \]
      respectively. Recall that our floor terms round towards zero. We have now
      converted \(w\) to normal form. So the normal form version of \(w\) is
      trivial if and only if all of the exponents of the \(a_i\)s, \(b_i\)s,
      \(d_i\)s in its normal form are equal to \(0\). That is, for all valid
      \(m\), the following system of equations hold. As each of the following
      equations is computed by setting the exponent of a generator to \(0\), we
      give the generator responsible for each equation in brackets next to the
      equation. This corresponds to going from
      (\ref{third_malcev_torsion_case_eqn}) to (\ref{torsion_case_Z_system}) in
      Example \ref{torsion_and_csts_ex}. We given in brackets at the left of
      each equation the generator that is being equated to zero to obtain this
      equation. Recall again we are using \(d_0\) to represent \(c\).
      \begin{align*}
        (a_m) \quad & \sum_{t = 1}^N \omega_{t, a_m}
        + \sum_{t = 1}^N \epsilon_t \boldsymbol{\nu}_{p_t, a_m} = 0, \\
        (b_m) \quad & \sum_{t = 1}^N \omega_{t, b_m}
        + \sum_{t = 1}^N \epsilon_t \boldsymbol{\nu}_{p_t, b_m} \equiv_{l_m} 0, \\
        (d_m) \quad & \sum_{t = 1}^N \omega_{t, d_m}
        + \sum_{t = 1}^N \epsilon_t \boldsymbol{\nu}_{p_t, d_m}
        \\
        & - \mathop{\sum_{t, u = 1}^N}_{u < t} \left(
        \mathop{\sum_{i, j = 1}^n}_{i < j}
        \omega_{t, a_i} \alpha_{ijm} (\omega_{u, a_j} + \epsilon_u \boldsymbol{\nu}_{p_u, a_j} )
        + \sum_{i = 1}^n \sum_{j = 1}^r \omega_{t, a_i} \gamma_{ijm}
        (\omega_{u, b_j} + \epsilon_u \boldsymbol{\nu}_{p_u, b_j}) \right) \\
        & - \mathop{\sum_{t, u = 1}^N}_{u \leq t}
        \left(\mathop{\sum_{i, j = 1}^n}_{i < j} \epsilon_t \boldsymbol{\nu}_{p_t, a_i} \omega_{u, a_j}
        \alpha_{ijm}
        + \sum_{i = 1}^n \sum_{j = 1}^r \epsilon_t \boldsymbol{\nu}_{p_t, a_i} \omega_{u, b_j}
        \gamma_{ijm} \right) \\
        \nonumber
        & - \mathop{\sum_{t, u = 1}^N}_{u < t} \left(
        \mathop{\sum_{i, j = 1}^n}_{i < j} \epsilon_t \epsilon_u \boldsymbol{\nu}_{p_t, a_i} \boldsymbol{\nu}_{p_u, a_j}
        \alpha_{ijm} +
        \sum_{i = 1}^n \sum_{j = 1}^r \epsilon_t \epsilon_u \boldsymbol{\nu}_{p_t, a_i} \boldsymbol{\nu}_{p_u, b_j} \gamma_{ijm}
        \right) \\
        & - \mathop{\sum_{t, u = 1}^N}_{u < t} \left(
        \mathop{\sum_{i, j = 1}^r}_{i < j}
        \omega_{t, b_i} \beta_{ijm} (\omega_{u, b_j} + \epsilon_u \boldsymbol{\nu}_{p_u, b_j} )\right)
        - \mathop{\sum_{t, u = 1}^N}_{u \leq t}
        \mathop{\sum_{i, j = 1}^r}_{i < j} \epsilon_t \boldsymbol{\nu}_{p_t, b_i} \omega_{u, b_j}
        \beta_{ijm}
        \\
        & - \mathop{\sum_{t, u = 1}^N}_{u < t}
        \mathop{\sum_{i, j = 1}^r}_{i < j} \epsilon_t \epsilon_u \boldsymbol{\nu}_{p_t, b_i} \boldsymbol{\nu}_{p_u, b_j} \beta_{ijm}
        - \sum_{t = 1}^N \sum_{i = 1}^r \eta_{im} \left \lfloor
        \frac{\omega_{t, b_i} + \epsilon_t \boldsymbol{\nu}_{p_t, b_i}}{l_i} \right \rfloor
         \\
        & - \sum_{z = 1}^N  \delta_{\epsilon_z} \left(
        \mathop{\sum_{i, j = 1}^n}_{i < j} \alpha_{ijm} \boldsymbol{\nu}_{p_z, a_i} \boldsymbol{\nu}_{p_z, a_j}
        + \sum_{i = 1}^n \sum_{j = 1}^r \gamma_{ijm}
        \boldsymbol{\nu}_{p_z, a_i} \boldsymbol{\nu}_{p_z, b_j}
        + \mathop{\sum_{i, j = 1}^r}_{i < j} \beta_{ijm} \boldsymbol{\nu}_{p_z, b_i}
        \boldsymbol{\nu}_{p_z, b_j} \right)\equiv_{k_m} 0
      \end{align*}
      Replacing constants in these equations with the constants stated in the
      lemma completes the proof.
    \end{proof}

  We use the following definitions to restate Lemma
  \ref{system_over_Z_technical_lem} in an easier format.

  \begin{defi}
    A \textit{quadratic function} from \(\mathbb{Z}^n\) to \(\mathbb{Z}\), where
    \(n \in \mathbb{Z}_{> 0}\), is a function \(f \colon \mathbb{Z}^n
    \to \mathbb{Z}\) such that there exist \(a_{ij} \in \mathbb{Z}\) for each
    \(i, j \in \{1, \ \ldots, \ n\}\) and \(b_1,
    \ \ldots, \ b_n, \ c \in \mathbb{Z}\), such that for all \((x_1, \ \ldots, \
    x_n) \in \mathbb{Z}^n\),
    \[
      f(x_1, \ \ldots, \ x_n) = \sum_{i = 1}^n \sum_{j = 1}^n
      a_{ij} x_i x_j + \sum_{i = 1}^m b_i x_i + c.
    \]
    A \textit{linear function} from \(\mathbb{Z}^n \to \mathbb{Z}\) is a
    function \(f \colon \mathbb{Z}^n \to \mathbb{Z}\), such that there exist
    \(b_1, \ \ldots, \ b_n, \ c \in \mathbb{Z}\), such that for all \((x_1, \
    \ldots, \ x_n) \in \mathbb{Z}^n\),
    \[
      f(x_1, \ \ldots, \ x_n) = \sum_{i = 1}^m b_i x_i + c.
    \]
  \end{defi}

  \begin{defi}
    Let \(w = 1\) be an equation in a class \(2\) nilpotent group. The system
    of equations in the ring of integers obtained by equating the exponents
    in the Mal'cev normal form for the equation to zero (that is, the system
    obtained in Lemma \ref{system_over_Z_technical_lem}) is called the
    \textit{\(\mathbb{Z}\)-system} of \(w = 1\).
  \end{defi}

  We now restate Lemma \ref{system_over_Z_technical_lem} up to grouping
  constants, and renaming constants and variables.

  \begin{lem}
    \label{Z_system_lem}
    The \(\mathbb{Z}\)-system of a single equation \(w = 1\) in a class \(2\)
    nilpotent group \(G\) with a virtually cyclic commutator subgroup is
    equivalent to a finite system of linear equations and congruences in
    \(\mathbb{Z}\), together with the following equations and congruences for
    finitely many \(k\):
    \begin{align}
      \label{quad_eqn}
      & \sum_{i = 1}^n - \alpha_{i} Y_{i} + f(X_1, \ \ldots, \ X_m)
      + \sum_{i = 1}^m \epsilon_{i} \left \lfloor \frac{\beta_{i} X_i +
      \kappa_{i}}{\gamma_{i}} \right \rfloor = 0, \\
      \label{quad_cong_with_quad_eqn}
      & g_k(X_1, \ \ldots, \ X_m)
      + \sum_{i = 1}^m \zeta_{ki} \left \lfloor \frac{\mu_{ki} X_i +
      \chi_{ki}}{\lambda_{ki}} \right \rfloor \equiv 0 \mod \delta_k,
    \end{align}
    where the values with Greek alphabet names are all constants computable from
    the class \(2\) nilpotent group and the single equation, \(X_1, \ \ldots, \
    X_m, \ Y_{1}, \ \ldots, \ Y_{n}\) are variables, and the
    \(f\) and the \(g_k\)s are quadratic functions.
  \end{lem}

  Lemma \ref{Z_system_decidable_lem} follows with a little work from the result
  of Siegel that the satisfiability of single quadratic equations in the ring of
  integers is decidable \cite{Siegel}. We refer the reader to
  \cite{duchin_liang_shapiro} for the proof.

  \begin{lem}[\cite{duchin_liang_shapiro}, Section 2.2]
    \label{Z_system_decidable_lem}
    It is decidable whether a system of equations of the form stated in
    Lemma \ref{Z_system_lem} admits a solution.
  \end{lem}

\section{Equations in virtually nilpotent groups}
  \label{finite_ext_sec}
  Within this section, we look at how equations behave when passing to a finite
  index overgroup. From \cite{duchin_liang_shapiro}, we know that the single
  equation problem is decidable in any class \(2\) nilpotent group with a
  virtually cyclic commutator subgroup. Doing so requires an understanding of
  how automorphisms of such groups behave. We start by investigating these.

  %

  \begin{prop}
    \label{automorphisms_prop}
    Let \(G\) be a class \(2\) nilpotent group with a virtually cyclic
    commutator subgroup. Let \(\{a_1, \ \ldots, \ a_n, \ b_1, \ \ldots, \ b_r, \
    c, \ d_1, \ \ldots, \ d_t\}\) be a Mal'cev generating set for \(G\). For
    each \(\theta \in \Aut(G)\), there exist linear functions \(f_1, \ \ldots, \
    f_{n + r} \colon \mathbb{Z}^{n + r} \to \mathbb{Z}\), linear functions
    \(g_0, \ g_1, \ \ldots, \ g_t \colon \mathbb{Z}^{t + 1} \to \mathbb{Z}\),
    and quadratic functions \(h_0, \ h_1, \ \ldots, \ h_t \colon \mathbb{Z}^{n +
    r} \to \mathbb{Z}\), such that for all Mal'cev normal form words \(w =
    a_1^{i_1} \cdots a_n^{i_n} b_1^{i_{n + 1}} \cdots b_r^{i_{n + r}} c^{q_0}
    d_1^{q_1} \cdots d_t^{q_t}\),
    \[
      \theta(w) =
      a_1^{f_1(i_1, \ \ldots, \ i_{n + r})} \cdots b_r^{f_{n + r}(i_1, \ \ldots,
      \ i_{n + r})} c^{g_0(q_0, \ \ldots, \ q_t) + h_0(i_1, \ \ldots, \ i_{n +
      r})}
      \cdots d_t^{g_t(q_0, \ \ldots, \ q_t) + h_t(i_1, \ \ldots, \ i_{n +
      r})}.
    \]
  \end{prop}

  \begin{proof}
    First note that the commutator subgroup is preserved by \(\theta\). Thus (taking
    \(d_0 = c\) by convention), we have for all \(\rho \in \{1, \ \ldots, \
    n\}\), \(\sigma \in \{1, \ \ldots, \ r\}\) and \(\varsigma \in \{0, \
    \ldots, \ t\}\),
    \begin{align*}
      \theta(a_\rho) & = a_1^{A_{\rho 1}}
      \cdots a_n^{A_{\rho n}} b_1^{B_{\rho 1}} \cdots b_r^{B_{\rho n}}
      d_0^{D_{\rho 0}} d_1^{D_{\rho 1}} \cdots d_t^{D_{\rho t}} \\
      \theta(b_\sigma) & = a_1^{A_{(n + \sigma)1}}
      \cdots a_n^{A_{(n + \sigma)n}} b_1^{B_{(n + \sigma)1}} \cdots b_r^{B_{(n + \sigma)r}}
      d_0^{D_{(n + \sigma)0}} d_1^{D_{(n + \sigma)1}} \cdots d_t^{D_{(n +
      \sigma)t}} \\ \theta(d_\varsigma) & = d_0^{D'_{\varsigma 0}}
      d_1^{D'_{\varsigma 1}} \cdots d_t^{D'_{\varsigma t}}.
    \end{align*}
    for some \(A_{11}, \ \ldots, \ A_{(n + r)n}, \ B_{11}, \ \ldots, \ B_{(n +
    r)r}, \ D_{10}, \ \ldots, \ D_{(n + r)t}, \ D'_{00}, \ \ldots, \ D'_{tt} \in
    \mathbb{Z}\). Let \(g = \theta(a_1^{i_1} \cdots a_n^{i_n} b_1^{i_{n + 1}}
    \cdots b_r^{i_{n + r}} d_0^{q_0} d_1^{q_1} \cdots d_t^{q_t})\). Then
    \begin{align*}
      g & =
      \prod_{\rho = 1}^{n + r} a_1^{i_\rho A_{\rho 1}}
      \cdots a_n^{i_\rho A_{\rho n}} b_1^{i_\rho B_{\rho 1}} \cdots b_r^{i_\rho B_{\rho n}}
      i_\rho d_0^{D_{\rho 0}} \cdots i_\rho d_m^{D_{\rho m}}
      \prod_{\varsigma = 0}^t d_0^{q_\varsigma D'_{\varsigma 0}} \cdots
      d_t^{q_\varsigma D'_{\varsigma t}} \\
      & = a_1^{\displaystyle \sum_{\rho = 1}^{n + r} i_\rho A_{\rho 1}}
      \cdots b_r^{\displaystyle \sum_{\rho = 1}^{n + r} i_\rho B_{\rho r}} \\
      & \ \phantom{=} \ \prod_{\varsigma = 0}^t
      d_\varsigma^{\displaystyle \sum_{\rho = 1}^{n + r} i_\rho D_{\rho \varsigma}
      + \sum_{m = 0}^t  q_m D'_{m \varsigma}
      + \mathop{\sum_{\rho, \rho' = 1}^n}_{\rho < \rho'} \mathop{\sum_{\sigma, \sigma' = 1}^{n + r}}_{\sigma < \sigma'} \pi_{\rho \rho' \varsigma} A_{\sigma \rho} A_{\sigma' \rho'} i_\sigma i_{\sigma'}
      + \sum_{\rho = 1}^n \sum_{\rho' = 1}^r \mathop{\sum_{\sigma, \sigma' = 1}^{n + r}}_{\sigma < \sigma'} \upsilon_{\rho
      \rho' \varsigma} A_{\sigma \rho} B_{\sigma' \rho'} i_\sigma i_{\sigma'} }
      \\
      & \ \phantom{= \prod_{\varsigma = 0}^t d_\varsigma}^{
      \displaystyle + \mathop{\sum_{\rho,
      \rho' = 1}^r}_{\rho < \rho'} \mathop{\sum_{\sigma, \sigma' = 1}^{n + r}}_{\sigma < \sigma'} \tau_{\rho \rho'
      \varsigma} B_{\sigma \rho} B_{\sigma' \rho'} i_\sigma i_{\sigma'}}
    \end{align*}
    If \(\sigma \in \{1, \ \ldots, \ n\}\), define \(f_\sigma(i_1, \ \ldots, \
    q_t) = \sum_{\rho = 1}^{n + r} i_\rho A_{\rho \sigma}\), and if \(\sigma \in
    \{n + 1, \ \ldots, \ n + r\}\), let \(f_\sigma(i_1, \ \ldots, \ q_t) = \sum_{\rho =
    1}^{n + r} i_\rho B_{\rho \sigma}\). For \(\varsigma \in \{0, \ \ldots, \
    t\}\), define functions \(g_\varsigma\) and \(h_\varsigma\) by
    \begin{align*}
      g_\varsigma(q_0, \ \ldots, \ q_t) & = \sum_{m = 0}^t  q_m D'_{m \varsigma} \\
      h_\varsigma(i_1, \ \ldots, \ i_{n + r}) & =
      \sum_{\rho = 1}^{n + r} i_\rho D_{\rho \varsigma}
      + \mathop{\sum_{\rho, \rho' = 1}^n}_{\rho < \rho'}
      \mathop{\sum_{\sigma, \sigma' = 1}^{n + r}}_{\sigma < \sigma'} \pi_{\rho \rho' \varsigma} A_{\sigma \rho} A_{\sigma' \rho'} i_\sigma i_{\sigma'}
      + \sum_{\rho = 1}^n \sum_{\rho' = 1}^r \mathop{\sum_{\sigma, \sigma' = 1}^{n + r}}_{\sigma < \sigma'} \upsilon_{\rho
      \rho' \varsigma} A_{\sigma \rho} B_{\sigma' \rho'} i_\sigma i_{\sigma'}\\
      & \ \phantom{=} \
      + \mathop{\sum_{\rho, \rho' = 1}^r}_{\rho < \rho'}
       \mathop{\sum_{\sigma, \sigma' = 1}^{n + r}}_{\sigma < \sigma'} \tau_{\rho \rho'
      \varsigma} B_{\sigma \rho} B_{\sigma' \rho'} i_\sigma i_{\sigma'}.
    \end{align*}
    We have that \(\theta(g)\) equals
    \[
      a_1^{f_1(i_1, \ \ldots, \ i_{n + r})} \cdots b_r^{f_{n + r}(i_1, \ \ldots,
      \ i_{n + r})} d_0^{g_0(q_0, \ \ldots, \ q_t) + h_0(i_1, \ \ldots, \ i_{n +
      r})} \cdots d_t^{g_t(q_0, \ \ldots, \ q_t) + h_t(i_1, \ \ldots, \ i_{n +
      r})}.
    \]
    Moreover, the functions \(f_\sigma\) and \(g_\varsigma\) are linear, and
    the functions \(h_\varsigma\) are quadratic, as required.
  \end{proof}

  We generalise an equation in a group \(G\) to allow variables to be acted upon
  by automorphisms of \(G\). As we will see in
  Lemma \ref{finite_extension_twisted_lem}, solving twisted equations in \(G\) is
  `equivalent' to solving equations in finite extensions of \(G\).

  \begin{defi}
    Let \(G\) be a group. A \textit{twisted equation}
    in \(G\) with \textit{variables} \(V\) is an element
    \(w \in (G \cup F(V) \times \Aut(G))^\ast\),
    and is again denoted \(w = 1\). Define the function
    \begin{align*}
      p \colon G \times \Aut(G) & \to G \\
      (g, \ \psi) & \mapsto g \psi.
    \end{align*}
    If \(\phi \colon F(V) \to G\) is a homomorphism, let \(\bar{\phi}\) denote
    the (monoid) homomorphism from \((G \cup F(V) \times \Aut(G))^\ast\) to \((G
    \times \Aut(G))^\ast\), defined by \((h, \ \psi) \bar{\phi} = (h \phi, \
    \psi)\) for \((h, \ \psi) \in F(V) \times \Aut(G)\) and \(g \bar{\phi} = g\)
    for all \(g \in G\). A \textit{solution} to \(w = 1\) is a homomorphism
    \(\phi \colon F(V) \to G\), such that \(w \bar{\phi} p = 1_G\).

    For the purposes of decidability, in finitely generated groups, the elements
    of \(G\) will be represented as words over a finite generating set, and in
    twisted equations, automorphisms will be represented by their action on the
    generators.

    The \textit{single twisted equation problem} in \(G\) is the decidability
    question as to whether there is a terminating algorithm that accepts a
    twisted equation \(w = 1\) as input, returns \textsc{yes} if \(w =
    1\) admits a solution and \textsc{no} otherwise, where elements of \(G\)
    within \(w\) are represented by words over a finite generating set, and
    automorphisms are represented by their action on the finite generating set.
  \end{defi}

  We give a brief example of a twisted equation in \(\mathbb{Z}\).

  \begin{exa}
    Consider the twisted equation \(X = Y \psi\) in the group \(\mathbb{Z}\)
    with the generator \(a\), where \(\psi \in \Aut(\mathbb{Z})\) maps \(a\) to
    \(a^{-1}\) (the unique non-identity automorphism). It follows that this equation
    is equivalent to \(X = Y^{-1}\), which is not difficult to show has the
    solution set
    \[
      \{(a^x, \ a^{-x}) \mid x \in \mathbb{Z}\}.
    \]
    More generally, any twisted equation in \(\mathbb{Z}\) can be solved using
    this argument, as the identity automorphism can simply be removed without
    affecting the solution set, and the non-identity automorphism can be
    replaced by adding an inverse sign to the variable it acts on. This yields
    an (untwisted) equation in \(\mathbb{Z}\).
  \end{exa}

  The following lemma is widely known, although often not stated explicitly.
  Variations of it have been used to show systems of equations in virtually free
  groups, or virtually abelian groups are decidable, or to describe the
  structure of solution sets (see for example \cites{dahmani_guirardel,
  VF_eqns_EDT0L, VAEP}). We include a proof for completeness.

  \begin{lem}
    \label{finite_extension_twisted_lem}
    Let \(G\) be a group with a finite-index normal subgroup \(H\), such that
    \(H\) has decidable single twisted equation problem. Then \(G\) has
    decidable single equation problem.
  \end{lem}

  \begin{proof}
    Let \(T\) be a (finite) transversal for \(H\). consider an equation \(w
    = 1\) in \(G\). We can express every element in \(G\) in the form \(ht\) for
    \(h \in H\) and \(t \in T\). Thus we can write \(w = 1\) as
    \begin{align}
      \label{finite_ext_eqn}
      h_1 t_1 X_{i_1}^{\epsilon_1} \cdots h_K t_K X_{i_K}^{\epsilon_K} = 1,
    \end{align}
    where \(h_j \in H\), \(t_j \in T\), and \(\epsilon_j \in \{-1, \ 1\}\), for
    all \(j\), and \(X_1, \ \ldots, \ X_N\) are the variables of \( = 1\).
    If \((g_1, \ \ldots, \ g_N)\) is a solution, then each \(g_j\) can be
    expressed in the normal subgroup-transversal form, and so by applying this
    fact to our variables, (\ref{finite_ext_eqn}) admits a solution if and only
    if the following equation does:
    \begin{align}
      \label{finite_ext_subgp_trans_eqn}
      h_1 t_1 (Y_{i_1} Z_{i_1})^{\epsilon_1} \cdots h_K t_K (Y_K
      Z_K)^{\epsilon_K} = 1,
    \end{align}
    where \(X_j = Y_j Z_j\), \(Y_j\) is a variable over \(H\), and \(Z_j\) is a
    variable over \(T\), for all \(j\). For each \(g \in G\), define \(\psi_g
    \colon G \to G\) by \(h \psi_g =  g h g^{-1}\). As \(H\) is normal, these
    automorphisms fix \(H\). We will abuse notation, and extend this notation to
    define \(\psi_{Z_1}, \ \ldots, \ \psi_{Z_N}\). Let
    \begin{align*}
      \delta_j = \left\{
      \begin{array}{cl}
        0 & \epsilon_j = 1 \\
        1 & \epsilon_j = -1.
      \end{array}
      \right.
    \end{align*}
    Thus (\ref{finite_ext_subgp_trans_eqn}) is equivalent to
    \begin{align*}
      (Y_{i_1}^{\epsilon_1} \psi_{Z_{i_1}}^{\delta_1}) Z_{i_1}^{\epsilon_1}
      h_1 t_1 \cdots
      (Y_{i_K}^{\epsilon_K} \psi_{Z_{i_K}}^{\delta_K}) Z_{i_K}^{\epsilon_K}
      h_K t_K = 1.
    \end{align*}
    By pushing all \(Y_j\)s and \(h_j\)s to the left, we obtain
    \begin{align}
      \label{finite_ext_subgp_to_left_eqn}
      (Y_{i_1}^{\epsilon_1} \psi_{Z_{i_1}}^{\delta_1})
      (h_1 \psi_{Z_{i_1}}^{\epsilon_1}) \cdots
      (Y_{i_K}^{\epsilon_K} \psi_{Z_{i_K}}^{\delta_K} \psi_{t_{K - 1}}
      \psi_{Z_{i_{K - 1}}}^{\epsilon_{K - 1}} \cdots \psi_{t_1}
      \psi_{Z_{i_1}}^{\epsilon_1})
      (h_K \psi_{Z_{i_K}}^{\epsilon_K} \psi_{t_{K - 1}} \cdots \psi_{t_1}
      \psi_{Z_{i_1}}^{\epsilon_1})
      Z_{i_1}^{\epsilon_1} t_1 \cdots Z_{i_K}^{\epsilon_K} t_K = 1.
    \end{align}
    We have that a necessary condition for a potential solution \((y_1 z_1, \
    \ldots, \ y_N z_N)\) to (\ref{finite_ext_subgp_to_left_eqn}) to be a
    solution is that \(t_1 z_{i_1} \cdots t_K z_{i_K} \in H\). Let \(A\) be the
    set of tuples \((z_1, \ \ldots, \ z_N)\) of transversal elements such that
    \(t_1 z_{i_1} \cdots t_K z_{i_K} \in H\). As \(T\) is finite, so is \(A\),
    and so the solution set to (\ref{finite_ext_subgp_to_left_eqn}) is equal to
    the finite union across \((z_1, \ \ldots, \ z_N) \in A\) of the following
    twisted equations in \(H\):
    \begin{align*}
      (Y_{i_1}^{\epsilon_1} \psi_{z_{i_1}}^{\delta_1})
      (h_1 \psi_{z_{i_1}}^{\epsilon_1}) \cdots
      (Y_{i_K}^{\epsilon_K} \psi_{z_{i_K}}^{\delta_K} \psi_{t_{K - 1}}
      \psi_{z_{i_{K - 1}}}^{\epsilon_{K - 1}} \cdots \psi_{t_1}
      \psi_{z_{i_1}}^{\epsilon_1})
      (h_K \psi_{z_{i_K}}^{\epsilon_K} \psi_{t_{K - 1}} \cdots \psi_{t_1}
      \psi_{z_{i_1}}^{\epsilon_1})
      z_{i_1}^{\epsilon_1} t_1 \cdots z_{i_K}^{\epsilon_K} t_K = 1
    \end{align*}
    Since the twisted single equation problem in \(H\) is decidable, and we can
    check if each of these equations admit solutions, noting there are finitely
    many of them. If at least one admits a solution, then \(w = 1\) does.
    If none admit a solution, then neither does \(w = 1\).
  \end{proof}

  Now that we have Lemma \ref{finite_extension_twisted_lem}, the following is
  (almost) all that is required to prove that the single equation problem in a
  virtually class \(2\) nilpotent group with a virtually cyclic commutator
  subgroup is decidable.

  \begin{lem}
    \label{class_2_twisted_lem}
    The single twisted equation problem in a class \(2\) nilpotent group
    with a virtually cyclic commutator subgroup is decidable.
  \end{lem}

  \begin{proof}
    Consider a single twisted equation \(\mathcal E\) in a class \(2\) nilpotent
    group \(G\). We can view \(\mathcal E\) by applying the automorphisms to the
    words \(\boldsymbol{\nu}_z\) within the statement of Lemma
    \ref{system_over_Z_technical_lem}. Using Proposition
    \ref{automorphisms_prop}, automorphisms act as linear functions of
    \(\boldsymbol{\nu}_{z, a_1}, \ \ldots, \ \boldsymbol{\nu}_{z, a_n}\) and
    \(\boldsymbol{\nu}_{z, b_1}, \ \ldots, \ \boldsymbol{\nu}_{z, \ b_r}\), and
    quadratic functions of \(\boldsymbol{\nu}_{z, \ c}\) and
    \(\boldsymbol{\nu}_{z, \ d_1}, \ \ldots, \ \boldsymbol{\nu}_{z, \ d_t}\).

    In Lemma \ref{system_over_Z_technical_lem}, the values
    \(\boldsymbol{\nu}_{z, \ c}\) and \(\boldsymbol{\nu}_{z, \ d_1}, \ \ldots, \
    \boldsymbol{\nu}_{z, \ d_m}\) only appear in linear terms in the system
    (that is, they never appear in the form \(\boldsymbol{\nu}_{z, c}
    \boldsymbol{\nu}_{z, \ d_i})\). Thus after applying the automorphisms, we
    will have a system of the form stated in Lemma \ref{Z_system_lem} equivalent
    to \(\mathcal E\). By Lemma \ref{Z_system_decidable_lem}, a system of the
    form of Lemma \ref{Z_system_lem} is decidable, and thus the result follows.
  \end{proof}

  All that remains to prove Theorem \ref{virt_class_2_decidable_thm} is to deal
  with the difference between a finite-index subgroup and a finite-index normal
  subgroup.

  \begin{lem}
    \label{finite_extension_to_virtually_lem}
    Let \(N\) be a finite-index normal subgroup of a group \(H\), such that
    \(H\) is nilpotent of class \(2\), and \(H\) has a virtually cyclic
    commutator subgroup. Then \(N\) is nilpotent of class \(2\) and has a
    virtually cyclic commutator subgroup.
  \end{lem}

  \begin{proof}
    Subgroups of nilpotent groups of class \(c\) are always nilpotent (see, for
    example \cite{theory_of_nilpotent_groups}, Theorem 2.4), and thus \(N\) is
    nilpotent of class \(2\). Moreover \([N, \ N] \leq [H, \ H]\), and so
    \([N, \ N]\) is contained in a virtually cyclic group, and is therefore
    virtually cyclic.
  \end{proof}

  Combining our lemmas now gives the following.

\mainthm*

  \begin{proof}
    Let \(\mathcal{P}\) be the property of being class \(2\) nilpotent and
    having a virtually cyclic commutator subgroup. By taking the normal core, we
    have that a virtually \(\mathcal{P}\) group admits a finite-index normal
    subgroup. Lemma \ref{finite_extension_twisted_lem} implies that this normal
    subgroup must be \(\mathcal{P}\). We have therefore shown that any virtually
    \(\mathcal{P}\) group has a finite-index normal subgroup that is
    \(\mathcal{P}\). We have from Lemma \ref{class_2_twisted_lem} that the
    single twisted equation problem in a group with \(\mathcal{P}\) is
    decidable. The result now follows by
    \ref{finite_extension_to_virtually_lem}.
  \end{proof}

\section*{Acknowledgements}

  The author would like to thank Laura Ciobanu for lots of helpful mathematical
  discussions and writing advice. The author would also like to thank Alex
  Evetts for explaining some results about virtually class \(2\) nilpotent
  groups, Alan Logan for writing advice, and Albert Garreta for help finding
  references and useful comments. I would like to thank the anonymous reviewers
  for very helpful comments. I would like to acknowledge the support provided by
  the London Mathematical Society, the Heilbronn Institute for Mathematical
  Research and the University of St Andrews during the writing of this paper.

\bibliography{references}

\end{document}